\documentclass[a4paper,11pt]{article}
\usepackage{latexsym,amssymb,amsmath,amsthm,amsxtra,xypic}
\usepackage{stmaryrd}
\usepackage{amsfonts}
\usepackage{amscd}
\usepackage{mathrsfs}
\usepackage[dvips]{graphicx}
\usepackage[all]{xy}

\newcommand{\allowpagebreak}
\allowdisplaybreaks

\newtheorem{thm}{Theorem}[section]

\newtheorem{pro}[thm]{Proposition}
\newtheorem{ex}[thm]{Example}

\theoremstyle{definition}
\newtheorem{defi}[thm]{Definition}

\newcommand {\emptycomment}[1]{}

\newcommand{\ot}{\otimes}

\def\LM{\mathcal{LM}}

\newcommand{\al}{\alpha}
\newcommand{\alin}[1]{\alpha(#1)}
\newcommand{\alm}{\alpha_M}

\newcommand{\alv}{\alpha_V}
\newcommand{\alw}{\alpha_W}
\newcommand{\alaw}{\alpha_{A\oplus W}}
\newcommand{\almv}{\alpha_{M\oplus V}}

\newcommand{\lam}{\lambda}

\newcommand{\trl}{\triangleleft}
\newcommand{\trr}{\triangleright}

\newcommand{\g}{{A}}

\newcommand{\wm}{\widehat{M}}
\newcommand{\wg}{\widehat{A}}
\newcommand{\wal}{\widehat{\alpha}}

\newcommand{\frkg}{{A}}

\newcommand{\fg}{{A}}

\newcommand{\frkh}{V}

\newcommand{\dlam}{_\lambda}
\newcommand{\cdlam}{\cdot_\lambda}

\newcommand{\dM}{f}
\newcommand{\E}{\mathrm{E}}

\newcommand{\Hom}{\mathrm{Hom}}

\newcommand{\Ker}{\mathrm{Ker}}

\newcommand{\Imm}{\mathrm{Im}}

\newcommand{\id}{\mathrm{id}}

\allowdisplaybreaks

\begin{document}

\title{Cohomology of Hom-associative algebras in Loday-Pirashvili category with applications}

\author{Tao Zhang\thanks{{E-mail:} zhangtao@htu.edu.cn}\\
{\footnotesize College of Mathematics, Henan Normal University, Xinxiang 453007, PR China}}

\date{}
\maketitle

\footnotetext{{\it{Keyword}: Hom-associative algebras, Hom-dialgebras, Loday-Pirashvili category, representation, cohomology, deformations.}}

\footnotetext{{\it{Mathematics Subject Classification (2020)}}: 16E40, 16S80, 16W99.}

\begin{abstract}
We introduce the concept of Hom-associative algebra structures in Loday-Pirashvili category.
The cohomology theory of Hom-associative algebras in this category is studied.
Some applications on deformation and abelian extension theory are given.
We also introduce the notion of Nijenhuis operators to describe trivial deformations.
It is proved that equivalent classes of abelian extensions are one-to-one correspondence to the elements of the second cohomology groups.
\end{abstract}

%\tableofcontents

\section{Introduction}
A Hom-associative algebra is an algebra ${A}$ with an additional linear map $\al:{A}
\to {A}$  satisfying the Hom-associative identity:
$$\al(x)(yz)=(x y)\al(z),$$
for all $x, y, z\in {A}$.
The representations, abelian extensions, deformations  and cohomology theory of Hom-algebras were studied in \cite{AEM,CGM,CMM,HLS,MS1,MS2,Yau}.
It is known that abelian extensions and deformations of Hom-type algebras are governed by the second cohomology group.
There are many other types of Hom-structures including Hom-Hopf algebras, Hom-Lie-Rinehart algebras, Hom-Lie antialgebras and Hom-Lie-Yamaguti algebras, see \cite{Caen,ZZ1,ZZ2,Zhang,Zhang1}.

In the remarkable paper \cite{LP}, Loday and Pirashvili  introduced a tensor category $\LM$ of linear maps.
Roughly speaking, it is a category consists of linear maps  $f:V\to W$ as objects.
A morphism  between two objects $f:V\to W$ and $f':V' \to W'$ is a pair of linear maps $(\phi, \psi)$ where $\phi:V\to V', \psi: W\to W'$ such that $\psi\circ f=f'\circ \phi$.
Then they introduced the concept of associative algebra $(M,\g, f)$ in $\LM$.
It is a bimodule $M$ of an ordinary associative algebra ${A}$ such that  $f: M\to \frkg$ is a bimodule map:
\begin{eqnarray}
\label{cm000}
f(x\cdot m)=x f(m), \quad  f(m\cdot x)= f(m)x,
\end{eqnarray}
for all  $x\in \g$ and $m\in M$.
This Loday-Pirashvili  category provides a rich framework for studying various constructions related to Leibniz algebras. In particular, a Leibniz algebra becomes a Lie object in this tensor category, and the universal enveloping algebra of Leibniz algebras becomes a cocommutative Hopf algebra in this tensor category. In \cite{Kur}, Kurdiani  constructs a cohomology theory for Lie algebras in $\LM$.
Unfortunately, the representation and cohomology theory of this associative algebra $(M,\g, f)$ in $\LM$ is not established in Loday and Pirashvili's paper \cite{LP}.
This is the main motivation of this paper. We will fill this gap by studying the Hom-associative algebras in Loday-Pirashvili  category

%He show that abelian extensions of Lie algebras in $\LM$ are classified by the second cohomology group.
%In \cite{Rov}, Rovi studied Lie algebroids in the Loday-Pirashvili  category.

In this paper, we introduced the concept of Hom-associative algebras in $\LM$ as follows.
Let $(\g,\al)$ be a Hom-associative algebra, $(M,\alm)$ be a bimodule of $(\g,\alpha )$.
Then a Hom-associative algebra in $\LM$ is a bimodule map $f: M\to \frkg$ such that the following conditions hold:
\begin{eqnarray}
\label{cm001}
\al\circ f(m)&=&f\circ\alm(m),\\
\label{cm002}
f( x\cdot m )&=&x f(m), \quad  f( m\cdot x )= f(m)x,
\end{eqnarray}
for all  $x\in \g$ and $m\in M$.
When $\al$ and $\alm$ are identity maps, this is just the Loday-Pirashvili's associative algebra in $\LM$.
We will show that this new concept of Hom-associative algebras in $\LM$ is closely related to the concept of Hom-dialgebras.
%Now it is natural a question what is the representation and cohomology theory for Hom-associative algebras in $\LM$?
%In this paper, we will  give an answer to this question.
We investigate the representation and cohomology theory for this type of Hom-associative algebras.
As applications, we will study the infinitesimal deformation and abelian extension theory in detail.

The organization of this paper is as follows.
In Section 2, we review some basic facts about Hom-associative algebras and Hom-dialgebras.
In Section 3, we introduced the concept of Hom-associative algebras in $\LM$ and study its elementary properties.
In Section 4, we investigate the representation and cohomology theory of Hom-associative algebras in $\LM$.
We show that given a Hom-associative algebra in $\LM$ and a bimodule,
we can construct a new a Hom-associative algebra in $\LM$ on their direct sum spaces which is called  the semidirect product.
Low dimensional cohomologies are given in detail in this section.
In Section 5, we study applications of low dimensional cohomolgy.
This include two subsections. In subsection 5.1, we investigate the infinitesimal deformations of Hom-associative algebras in $\LM$.
The notion of Nijenhuis operators is introduced to describe trivial deformations.
In the last subsection 5.2,  we prove that equivalent classes of abelian extensions are one-to-one correspondence to the elements of the second cohomology groups.

Throughout the rest of this paper, we work over a fixed field of characteristic 0.
A Hom-vector space is a pair $(V, \al_V)$ consisting of  a vector space $V$ and  a linear map $\al_V: V\to V$.
A direct sum of Hom-vector spaces  $(V, \al_V)$ and $(W, \al_W)$ is a direct sum of vector spaces $V\oplus W$ with linear map $\al_V\oplus\al_W: V\oplus W\to V\oplus W$ given by $(\al_V\oplus\al_W)(v+w)=\al_V(v)+\al_W(w)$.
The space of linear maps from a vector space  $V$ to $W$ is denoted by $\Hom(V,W)$.

\section{Preliminaries}
%\section{Hom-vector space  and Hom-associative algebras}

In this section, we recall some definitions and fix some notations about Hom-associative algebras.

A Hom-vector space is a pair $(M, \al_M)$ consisting of  a vector space $M$ and  a linear map $\al_M: M\to M$. A morphism $f: (M, \al_M) \to (N, \al_N) $ of Hom-vector spaces is a linear map $f : M\to N$ such that $f\circ \al_M = \al_N\circ f$.

\begin{defi}[\cite{MS1}]
A Hom-associative algebra is a triple $(\g,\alpha)$ where $\g$ a vector space equipped with
a multiplication $\g\otimes \g\to\g: x\otimes y\mapsto xy$  and a linear map $\alpha:\g\to \g$ satisfying the following
Hom-associative identity:
$$\al(x)(yz)=(x y)\al(z),$$
for all $x,y,z\in\g$.
A Hom-associative algebra is called a multiplicative Hom-associative algebra if $\alpha$ is an algebraic homomorphism, i.e, for all $x,y\in\frkg$, $\alpha(xy)= \alpha(x) \alpha(y).$
\end{defi}

In the following of this paper, we always assume our Hom-algebras are multiplicative.

A homomorphism between two Hom-associative algebras $(\g,\alpha)$ and $(\g',\alpha')$ is a linear map $\phi:\g\to\g'$ such that
\begin{equation}
\phi\circ\alpha(x)=\alpha'\circ\phi(x),\quad \phi(xy)=\phi(x)\phi(y),
\end{equation}
for all $x,y\in\g.$

\begin{defi}[\cite{Yau1}]
A Hom-dialgebra is a vector space $A$ equipped with two Hom-associative linear maps $\dashv, \vdash:{A}\otimes {A} \to {A}$ such that the following identity holds:
\begin{itemize}
\item[$\bullet$]{\rm(D1)}\quad  $ (x\dashv y)\dashv \al(z)  =  \al(x) \dashv (y \vdash z)$,
\item[$\bullet$]{\rm(D2)}\quad  $ (x\vdash y)\dashv \al(z)  = \al(x)\vdash (y\dashv z)$,
\item[$\bullet$]{\rm(D3)}\quad  $ (x\dashv y)\vdash \al(z)  =  \al(x) \vdash (y\vdash z)$.
\end{itemize}
for all  $x,y,z\in A$.
\end{defi}

It is proved in \cite{Yau1} that given a Hom-dialgebra $({A}, \dashv, \vdash)$, there is a Hom-Leibniz algebra with bracket $[x,y]=x\vdash y-y\dashv x$.
A homomorphism between two Hom-dialgebras $({A},\alpha_{A})$ and $({A}',\alpha_{A'})$ is a linear map $\phi:{A}\to{A}'$ such that  the following identity holds:
\begin{equation}
\phi\circ\alpha_L(x)=\alpha_{L'}\circ\phi(x),\quad \phi(x\dashv y)=\phi(x)\dashv' \phi(y), \quad \phi(x\vdash y)=\phi(x)\vdash' \phi(y),
\end{equation}
for all $x,y\in{A}.$

\begin{ex}[\cite{MS1}]
Given an associative algebra $\g$ and an algebraic homomorphism $\alpha:\g\to\g$, define $\cdot_\alpha:\g\otimes \g\to\g$ by
$$x\cdot_\alpha y=\alpha(xy),\ \ \ \forall x,y\in\g.$$
Then $(\g, \cdot_\alpha,\alpha)$ is a Hom-associative algebra.
\end{ex}

%\begin{defi}[\cite{HLS},\cite{Sheng}]
% A representation of a Hom-associative algebra $(\g,\alpha)$ on a Hom-vector space $(M,\al_M)$ is a linear map $\cdot:A\otimes M \to M$ such that:
% \begin{eqnarray}
% \label{rep01}
% \rho(\alpha(x))\circ\alpha_M&=&\alpha_M\circ\rho(x),\\
% \label{rep02}
% \rho(xy)\circ\alpha_M&=&\rho(\alpha(x))\circ\rho(y)-\rho(\alpha(y))\circ\rho(x),
% \end{eqnarray}
%for all $x,y\in\g.$
%\end{defi}

\begin{defi}
For a Hom-associative algebra $(\g,\al)$, a bimodule of $(\g,\al)$ is a Hom-vector space $(M,\alm)$ together with two linear maps (called the left module and right module) $\cdot:A\otimes M\to M$ and $\cdot:M\otimes A\to M$ satisfying
\begin{eqnarray}
\label{mod00}
\alm(x\cdot m)&=&\alin x\cdot\alm(m),\\
\label{mod11}
\alm(m\cdot x)&=&\alm(m)\cdot\al(x),\\
\label{mod01}
{\alin x\cdot (y\cdot m)}&=&(xy)\cdot\alm(m),\\
\label{mod02}
{\alm(m)\cdot (xy)}&=&(m \cdot x)\cdot \alin y,\\
\label{mod03}
{\alin x \cdot(m\cdot y)}&=&(x\cdot m)\cdot \alin y,
\end{eqnarray}
for all $x,y\in\g$, $m\in M$.
\end{defi}

\begin{pro}[\cite{MS2}]
  Given a Hom-associative algebra $(\g,\al)$ and a bimodule over $(M,\alm)$.  Define $\alpha\oplus\alpha_{M}:\frkg\oplus M\longrightarrow \frkg\oplus M$ and a multiplication by
  $$
(\alpha\oplus \alpha_{M})(x, m)=(\alpha(x),\alpha_{M}(m)),
 $$
 \begin{equation}
   (x, m)(y, n)=(xy,x\cdot n+m\cdot y).
 \end{equation}
 Then $(\frkg\oplus M,\alpha\oplus\alpha_{M})$ is a Hom-associative algebra,
 which we call the semidirect product of $(\frkg,\alpha)$ and $(M,\alm)$.
\end{pro}

\section{Hom-associative algebras in $\LM$}

In this section, we introduce the concept of Hom-associative algebras in $\LM$ and give its elementary properties.

\begin{defi} \label{defi:Lie 2}
A Hom-associative algebra in $\LM$ contains of a linear map $\dM: M\to \frkg$
where $(\g,\al)$ is a Hom-associative algebra, $(M,\alm)$ is a bimodule of $(\g,\alpha )$,
and $f$ is a bimodule map, that is the following conditions hold:
\begin{eqnarray}
\label{cm01}
\al\circ f(m)&=&f\circ\alm(m),\\
\label{cm02}
f( x\cdot m )&=&x f(m), \quad  f( m\cdot x )= f(m)x,
\end{eqnarray}
for all  $x\in \g$ and $m\in M$.
\end{defi}

\begin{pro}[\cite{Yau1}]\label{prop0}
For any Hom-associative algebra  $(M,\fg, f)$ in $\LM$, we have a Hom-dialgebra on $(M,\alm)$  with multiplication defined by
$$m\dashv n\triangleq m\cdot f(n),\quad m\vdash n\triangleq f(m)\cdot n,$$
for all $m,n\in M$,
\end{pro}

\begin{proof} The proof is given in \cite{Yau1}. We give it here for the reader's convenience.
We only verify the following condition (D1) holds
$$(m\dashv n)\dashv\alm(p)=\alm(m)\dashv (n\vdash p).$$
In fact, the left hand side is equal to
$$(m\dashv n)\dashv\alm(p)=(m\cdot f(n))\cdot f\alm(p)=(m\cdot f(n))\cdot \alm f(p).$$
and the right hand side   is equal to
$$\alm(m)\dashv (n\vdash p)=\alm(m)\cdot f (f(n)\cdot p)=\alm(m) \cdot  (f(n)f(p)).$$
Thus the two sides are equal to each other because $(M,\alm)$ is a right module of $(\fg,\al)$.
Similar computations show that  conditions (D2) and (D3) hold since $(M,\alm)$ is a bimodule of $(\fg,\al)$.
The proof is completed.
\end{proof}

\begin{pro}\label{prop1}
Let $(\g,\al)$ be a Hom-associative algebra and $(M,\alm)$ be a bimodule.
Then  $ f: M\rightarrow\g$ is a Hom-associative algebra in $\LM$
if and only if the maps $(\id,  f) : \g\rtimes M \to \g\rtimes\g$ is a homomorphism of Hom-associative algebras.
\end{pro}

\begin{proof}
To see when the map $(\id,  f) : \g\rtimes M \to \g\rtimes\g$ is a homomorphism of Hom-associative algebras, we first check that
\begin{eqnarray*}
(\id,  f)\circ\alpha_{\g\rtimes M}(x,m)&=&(\id,  f)(\alpha(x),\alpha_ M(m))\\
&=&(\alpha(x), f\alpha_ M(m))=(\alpha(x),\alpha_ M f(m)),\\
\alpha_{\g\rtimes\g}\circ(\id,  f)(x,m)&=&\alpha_{\g\rtimes\g}(x, f(m))=(\alpha(x),\alpha f(m)),
\end{eqnarray*}
thus $(\id,  f)\circ\alpha_{\g\rtimes M}=\alpha_{\g\rtimes\g}\circ(\id,  f)$ if and only if $\alpha\circ f(m)= f\circ\alpha_ M(m)$.
This is condition \eqref{cm01}.

On the other hand, since
\begin{eqnarray*}
&&{(\id,  f) (x, m) (x', m')} = (xx',  f(x\cdot m')+f(m\cdot x')),
\end{eqnarray*}
and
\begin{eqnarray*}
&&{(\id,  f)(x, m) (\id,  f)(x', m')}=(xx', x\cdot f (m')+ f(m)\cdot x'),
\end{eqnarray*}
thus the right hand side of above two equations are equal to each other if and only if $f( x\cdot m )=x f(m)$ holds, which is the condition \eqref{cm02}.
Therefore  $(\id,  f)$ is a homomorphism of Hom-associative algebras if and only if  the conditions \eqref{cm01} and \eqref{cm02} hold.
This complete the proof.
\end{proof}

%\begin{defi}
%A $Cat^1$-Hom-associative algebra $\left(\g_1,\g_0,s,t\right)$ consists of a Hom-associative algebra $\g_1$ together with a Hom-Lie subantialgebra $\g_0$ and the structural homomorphisms: $s, t: \g_1\to \g_0$ such that
%\begin{eqnarray}
%\label{cat01}&&s\circ i=t\circ i=id_{\g_0},
%\end{eqnarray}
%where $id_{\g_0}:\g_0\to \g_0$ are the identity maps,
%$\ker s$ and $\ker t$ are the kernel spaces of $s$ and $t$.
%\end{defi}

%\begin{pro}
%The categories $\mathbf{XLie}$ and $\mathbf{Cat^1Lie}$ are equivalent .
%\end{pro}
%\begin{ex}
%When $M=0$, this is an Hom-associative algebra.
%When $f=0$, this is a semi-direct product of $\g$ and $(M,\alm)$.
%\end{ex}

\begin{ex}
Let $(\g,\al)$ be a Hom-associative algebra, $(M,\alm)=(\g\otimes \g, \al\ot \al)$, $f:\g\otimes \g\to \g$ be the multiplication in $\g$. Define a bimodule of $(\g,\al)$  on $(\g\otimes \g, \al\ot \al)$ by
$$x\cdot (a\otimes b)=xa\otimes \alin b, \quad (a\otimes b)\cdot x=\al(a)\otimes bx.$$
Then we have
\begin{eqnarray*}
\al\circ f(a\otimes b)&=&\al(ab)=\al(a)\al(b)\\
&=&f(\alin a\otimes \alin b)=f\circ \al\ot \al(a\otimes b),\\
\end{eqnarray*}
and
\begin{eqnarray*}
f(x\cdot a\otimes b)&=&(xa) \alin b\\
&=&\alin x (ab)=x f(a\otimes b).
\end{eqnarray*}
Thus for any Hom-associative algebra $(\g,\al)$, we obtain a Hom-associative algebra $(\g\otimes \g, \g, f)$ in $\LM$.
In this case, by  Proposition \ref{prop0}, $(\g\otimes \g,\al\ot \al)$  becomes a Hom-dialgebra under the multiplication
\begin{eqnarray*}
x\otimes y\dashv a\otimes b&=&\al(x)\otimes y (ab),\\
x\otimes y\vdash a\otimes b&=&(xy)a\otimes \al(b).
\end{eqnarray*}
 for all $x\otimes y, a\otimes b\in \g\otimes \g$.
\end{ex}

Let $(M,\fg, f)$ and $(M',\fg', f')$ be two Hom-associative algebras in $\LM$. A morphism between them is a pair of maps
$\phi=(\phi_0,\phi_1)$ where $\phi_0:\g\to\g'$ is Hom-associative algebra homomorphism and $\phi_1:M\to M'$ is $\fg$-equivariant map such that
$$\phi_0(xy)=\phi_0(x)\phi_0(y),\quad \phi_1(x\cdot m)=\phi_0(x)\cdot \phi_1(m),\quad  f'\circ \phi_1=\phi_0\circ f.$$

\section{Representation and Cohomology}

In this section, the representation and cohomology theory of a Hom-associative algebra in $\LM$ are given.
\begin{defi}
Let $(M,\g, f)$ be a Hom-associative algebra in $\LM$.
A bimodule of $(M,\g, f)$ is an object $(V,W,\varphi)$ in $\LM$ such that the following conditions are satisfied:

(i)  $(V,\alv)$ and $(W,\alw)$  are bimodules of  $(\g,\al)$  respectively;
% \begin{eqnarray}
% \rho_{1}(\alpha(x))\circ\alpha_V&=&\alpha_V\circx\cdot ,\\
% \rho_{1}(xy)\circ\alpha_V&=&\rho_{1}(\alpha(x))\circ\rho_{1}(y)-\rho_{1}(\alpha(y))\circx\cdot ,
% \end{eqnarray}
%  \begin{eqnarray}
% \rho_{2}(\alpha(x))\circ\alpha_W&=&\alpha_W\circ\rho_{2}(x),\\
% \rho_{2}(xy)\circ\alpha_W&=&\rho_{2}(\alpha(x))\circ\rho_{2}(y)-\rho_{2}(\alpha(y))\circ\rho_{2}(x),
% \end{eqnarray}

(ii) $\varphi:(V,\alv)\to (W,\alw)$ is a bimodule map:
\begin{eqnarray}\label{deflm01}
\varphi(x\cdot v)=x\cdot \varphi(v),\quad \varphi(v\cdot x)= \varphi(v)\cdot x;
\end{eqnarray}

(iii) there exists linear maps $\trr: W\otimes M\to V$ and $\trl: M\otimes W\to V$ such that:
\begin{eqnarray}\label{deflm02}
\varphi (w\trr m)=w\cdot f(m),\quad\varphi (m\trl w)= f(m)\cdot w;
\end{eqnarray}

(iv) these bimodule structure satisfying the following compatibility conditions
\begin{eqnarray}
\label{deflm11}
\al(x) \cdot (w\trr m)=(x\cdot w)\trr \alm(m),\\
\label{deflm12}
 \al_W(w)\trr  (x\cdot m)= (w\cdot x)\trr \alm(m),\\
 \label{deflm13}
 (m\trl w)\cdot \al(x) = \alm(m)\trl (w\cdot x),\\
\label{deflm14}
(m\cdot x) \trl \alw(w) =\alm(m)\trl (x\cdot w),\\
\label{deflm15}
\al(x) \cdot (m\trl w)=(x\cdot m)\trl \alw(w),\\
\label{deflm16}
 \alw(w)\trr  (m\cdot x)= (w\trr m)\cdot \al(x),
\end{eqnarray}
where $x\in \g$, $m\in M$, $v\in V$ and  $w\in W$.
\end{defi}

The above two equations \eqref{deflm01} and \eqref{deflm02} imply that the following diagrams commute:
\begin{equation}
\begin{gathered}
 \xymatrixcolsep{3.5pc}
 \xymatrixrowsep{2.5pc}
    \xymatrix{
 \g \otimes V  \ar[r]^-{}    \ar[d]_{1 \otimes \varphi} & V  \ar[d]^{\varphi} \\
       \g\otimes W  \ar[r]^{}  &  W  }
\end{gathered}\quad
\begin{gathered}
 \xymatrixcolsep{3.5pc}
 \xymatrixrowsep{2.5pc}
    \xymatrix{
 V \otimes A  \ar[r]^-{}    \ar[d]_{ \varphi \otimes 1} & V  \ar[d]^{\varphi} \\
       W\otimes A  \ar[r]^{}  &  W  }
\end{gathered}
\end{equation}
%%%%%%%%%%%%%%%%%%
and
\begin{equation}
\begin{gathered}
 \xymatrixcolsep{3.5pc}
 \xymatrixrowsep{2.5pc}
    \xymatrix{
M \otimes W  \ar[r]^{\trr }    \ar[d]_{f \otimes 1} & V  \ar[d]^{\varphi} \\
      A\otimes W  \ar[r]^{}  &  W  }
\end{gathered}
\quad
%\label{module square 2}
\begin{gathered}
 \xymatrixcolsep{3.5pc}
 \xymatrixrowsep{2.5pc}
    \xymatrix{
 W\otimes M  \ar[r]^{\trl}    \ar[d]_{1 \otimes f} & V  \ar[d]^{\varphi} \\
      W\otimes A  \ar[r]^{}  &  W  }
\end{gathered}.
\end{equation}

For example, let $(V,W,\varphi)=(M,\g, f)$, then we get the adjoint representation of $(M,\g, f)$ on itself as follows:
$(M,\alm)$ and $(\g,\al)$ are bimodules of $(\g,\al)$ in a natural way,
 $\varphi=f$ is a bimodule map:
\begin{eqnarray}
f(x\cdot m)=xf(m)=x\cdot f(m),
\end{eqnarray}
and there exists  maps $\trl:M\otimes \g\to M, m\trl x\triangleq m\cdot x$ and $\trr:A\otimes M\to M, x\trr m\triangleq x\cdot m$ such that
\begin{eqnarray}
f(m\trl x)=f(m\cdot x)=f(m)x=f(m)\cdot x.
\end{eqnarray}

We construct semidirect product of a Hom-associative algebra  $(M,\g, f)$  and its bimodule $(V,W,\varphi)$.
\begin{pro}\label{pro:semidirectproduct}
  Given a bimodule of the Hom-associative algebra $(M,\g, f)$ on $(V,W,\varphi)$. Define on $(M\oplus V,\frkg\oplus W,\widehat{f}=f+\varphi)$
  the following maps
\begin{equation}
\left\{\begin{array}{rcl}
\widehat{f}(m+v)&\triangleq&\dM(m)+\varphi(v),\\
{(x+w)(x'+w')}&\triangleq&xx'+x\cdot w'+w\cdot x',\\
{(x+w){\cdot} (m+v)}&\triangleq& x\cdot m +x\cdot v+ w\trr m,\\
{(m+v){\cdot} (x+w)}&\triangleq& m\cdot x +v\cdot x+ m\trl w.
\end{array}\right.
\end{equation}
 Then $(M\oplus V,\frkg\oplus W, \widehat{f})$ is a  Hom-associative algebra in $\LM$, which is called the semidirect product of the  Hom-associative algebra of
 $(M,\g, f)$ and $(V,W,\varphi)$.
\end{pro}

\begin{proof}
First,  since the Hom-map on the direct sum vector space $M\oplus V$ are given by $\al_{M\oplus V}=\al_{M}+\al_{V}$, then we have
\begin{eqnarray*}
\al_{\frkg\oplus W}\circ \widehat{f}(m+v)&=&\al\circ f(m)+\alw\circ\varphi(v)\\
&=&f\circ\alm (m)+\varphi\circ\alv(v)\\
&=&\widehat{f}\circ\al_{M\oplus V}(m+v).
\end{eqnarray*}

Next, we verify that $\widehat{f}$ is a bimodule map. We have the equality
\begin{eqnarray}\label{semi01}
\widehat{f}{\big((x+w)\cdot (m+v)\big)}={(x+w)\widehat{f}(m+v)}.
\end{eqnarray}
The left hand side of \eqref{semi01} is
\begin{eqnarray*}
\widehat{f}{\big((x+w)\cdot (m+v)\big)}&=&\widehat{f}{\big( x\cdot m +x\cdot v+ w\trr m\big)}\\
&=&f( x\cdot m )+\varphi(x\cdot v)+\varphi (w\trr m),
\end{eqnarray*}
and the right hand side of \eqref{semi01} is
\begin{eqnarray*}
{(x+w)\widehat{f}(m+v)}&=&(x+w)(\dM(m)+\varphi(v))\\
&=&x\dM(m)+x\cdot \varphi(v)+w\cdot f(m).
\end{eqnarray*}

Similarly, we have the equality
\begin{eqnarray}\label{semi02}
\widehat{f}{\big((m+v)\cdot (x+w)\big)}={\widehat{f}(m+v)(x+w)}.
\end{eqnarray}
The left hand side of \eqref{semi02} is
\begin{eqnarray*}
\widehat{f}{\big( (m+v)\cdot (x+w)\big)}&=&\widehat{f}{\big(m\cdot x +v\cdot x+ m\trl w\big)}\\
&=&f(m\cdot x)+\varphi(v\cdot x)+\varphi (m\trl w),
\end{eqnarray*}
and the right hand side of \eqref{semi02} is
\begin{eqnarray*}
{\widehat{f}(m+v)(x+w)}&=&(\dM(m)+\varphi(v))(x+w)\\
&=&f(m) x+\varphi(v)\cdot x+f(m)\cdot w.
\end{eqnarray*}
Thus the two sides of  \eqref{semi01} and \eqref{semi02}  are equal to each other by conditions \eqref{deflm01} and \eqref{deflm02}.

Finally, we verify that $M\oplus V$ carries a bimodule structure over the Hom-associative algebra $A\oplus W$.
The left module condition is
\begin{eqnarray}\label{sembimod01}
\big((x+w)(x'+w')\big)\cdot\almv(m+v)=\alaw(x+w)\cdot\big((x'+w')\cdot(m+v)\big).
\end{eqnarray}
The left hand side of \eqref{sembimod01} is
\begin{eqnarray*}
&&\big((x+w)(x'+w')\big)\cdot\almv(m+v)\\
&=&\big(xx'+x\cdot w'+w\cdot x'\big)\cdot(\alm(m)+\alv(v))\\
&=&(xx')\cdot \alm(m)+(xx')\cdot\alv(v)+(x\cdot w')\trr \alm(m)+(w\cdot x')\trr \alm(m)
\end{eqnarray*}
and the right hand side of  \eqref{sembimod01} is
\begin{eqnarray*}
&&\alaw(x+w)\cdot\big((x'+w')\cdot(m+v)\big)\\
&=&(\al(x)+\alw(w))\cdot\big(x'\cdot m+x'\cdot v+w'\trr m\big)\\
&=&\al(x)\cdot (x'\cdot m)+\al(x)\cdot (x'\cdot v)+\al(x)\cdot (w'\trr m)+\alw(w)\trr (x'\cdot w).
\end{eqnarray*}
Thus the two side of \eqref{sembimod01} is equal to each other if and only if \eqref{deflm11} and \eqref{deflm12} hold.
Similar computations show that $M\oplus V$ carries a right module structure over the Hom-associative algebra $A\oplus W$.
 if and only if \eqref{deflm13} and \eqref{deflm14} hold, and it satisfy the bimodule conditions if and only if \eqref{deflm11}--\eqref{deflm16} hold.
This complete the proof.
\end{proof}

%\section{Cohomology}
Now we revisit the cohomology of Hom-associative algebras as follows.
Let $(\g, \al)$ be  a Hom-associative algebra $(V,\alv)$ be a bimodule.

A $k$-hom-cochain on $\frkg$ with values in $V$ is defined to be a
$k$-cochain $\omega\in \Hom(\otimes^k\frkg;V)$ such that it is compatible with
$\alpha$ and $\alv$ in the sense that $\alv\circ \omega=\omega\circ \alpha^{\otimes k}$, i.e.
$$
\alv(\omega(u_1,\cdots,u_k))=\omega(\alpha(u_1),\cdots,\alpha(u_k)).
$$
Denote by $ C^k_{\alpha,A}(\frkg;V)$ the set of $k$-hom-cochains.
Define $\delta:C^k_{\alpha,A}(\frkg;V)\longrightarrow C^{k+1}(\frkg;V)$ by setting
\begin{eqnarray*}
\delta \omega(u_1,\cdots,u_{k+1})&=&\alpha^{k-1}(u_1)\cdot \omega(u_2,\cdots,u_{k+1})\\
&&+\sum_{i=1}^k(-1)^{i}\omega(\alpha(u_1)\cdots,u_i u_{i+1},\cdots,\alpha(u_{k+1}))\\
  &&+(-1)^{k+1}\omega(u_1,\cdots,u_{k})\cdot\alpha^{k-1}(u_{k+1}).
\end{eqnarray*}
It is proved that $\delta\circ \delta=0$, thus one obtain a cohomology theory for Hom-associative algebra \cite{AEM}.

Next we develop a cohomology theory for $(M,\g, f)$.
Let  $(V,W,\varphi)$ be a bimodule of $(M,\g, f)$ and the $k$-cochian  $C^{k}((M,\g, f),(V,W,\varphi))$ to be the space:
\begin{eqnarray}
\Hom(\otimes^{k} \g,W)\oplus\Hom(\mathcal{A}^{k-1},V)\oplus \Hom(\mathcal{A}^{k-2},W).
\end{eqnarray}
where
\begin{eqnarray*}
 &&\Hom(\mathcal{A}^{k-1},V):=
 \bigoplus_{i=1}^k  \Hom(\underbrace{A\otimes\cdots\otimes A}_{i-1}\otimes M\otimes \underbrace{A\otimes\cdots\otimes A}_{k-i},V),\\
  &&\Hom(\mathcal{A}^{k-2},W):=
 \bigoplus_{i=1}^k  \Hom(\underbrace{A\otimes\cdots\otimes A}_{i-1}\otimes M\otimes \underbrace{A\otimes\cdots\otimes A}_{k-1-i},W)
 \end{eqnarray*}
is the direct sum of all possible tensor powers of $A$ and $M$ in which $A$ appears $k-1$ and $k-2$  times but $M$ appears only once for all $k\geq 2$.
For  $k=0, 1$, we define $\Hom(\mathcal{A}^{0},V)=\Hom(M,V)$,
$\Hom(\mathcal{A}^{1},V)=\Hom(A\otimes M,V)\oplus \Hom(M\otimes A,V)$,
$\Hom(\mathcal{A}^{0},W)=\Hom(M,W)$ and $\Hom(\mathcal{A}^{1},W)=\Hom(A\otimes M,W)\oplus \Hom(M\otimes A,W)$.

We also define the coboundary map by
\begin{eqnarray}
D(\omega,\mu,\nu,\theta)=(-\delta_1\omega, \delta_2(\mu+\nu)-h\omega, -l\omega+\varphi^\sharp(\mu+\nu)-\delta_3\theta),
\end{eqnarray}
where $\delta_1,\delta_2,\delta_3$ are  coboundary maps in Hom-associative algebra cohomology of $(\g,\al)$ with coefficient in $W$,
$\Hom(M,V)\cong M^*\otimes V$, $\Hom(M,W)\cong M^*\otimes W$ respectively,  $M^*$ is dual bimodule of $A$,
and the following maps:
\begin{eqnarray}
&&\varphi^\sharp: \Hom(\mathcal{A}^{k-1},V)\to \Hom(\mathcal{A}^{k-1},W),\\
&&h=(h_1, h_2): \Hom({A}^{k},W)\to \Hom(\mathcal{A}^{k},V),\\
&&l=(l_1, l_2): \Hom({A}^{k},W)\to \Hom(\mathcal{A}^{k-1},W)
\end{eqnarray}
by
\begin{eqnarray}
&&\varphi^\sharp(\mu)(x_1, \cdots, x_{k-1},m)=\varphi(\mu(x_1, \cdots, x_{k-1}, m)),\\
&&\varphi^\sharp(\nu)(m,x_1, \cdots, x_{k-1})=\varphi(\nu(m,x_1, \cdots, x_{k-1})),\\
&&h_1(\omega)(x_1, \cdots, x_{k},m)=\omega( x_1, \cdots, x_{k})\trr \al_{M}(m),\\
&&h_2(\omega)(m,x_1, \cdots, x_{k})=\al_{M}(m)\trl \omega( x_1, \cdots, x_{k}),\\
&&l_1(\omega)(m, x_1, \cdots, x_{k-1})=\omega(f(m),x_1, \cdots, x_{k-1}),\\
&&l_2(\omega)(x_1, \cdots, x_{k-1},m)=\omega(x_1, \cdots, x_{k-1}, f(m)).
\end{eqnarray}

We conjecture that this is a cochain complex $C^k((M,\g, f),(V,W,\varphi))$ whose cohomology group
$$H^k((M,\g, f),(V,W,\varphi))=Z^k((M,\g, f),(V,W,\varphi))/B^k((M,\g, f),(V,W,\varphi))$$
is defined as the cohomology group of $(M,\g, f)$ with coefficients in $(V,W,\varphi)$.
Since in this paper, we only need the second cohomology group, the proof of $D\circ D = 0$ for $k=2$ will given in detail as follows.

More precisely, the cochain complex is given by
\begin{equation} \label{eq:complex}
\begin{split}
 & \qquad \qquad\  W\stackrel{D_0}{\longrightarrow}\\
 &  \quad\Hom(\g,W)\oplus\Hom(M,V)\stackrel{D_1}{\longrightarrow}\\
 & \Hom(A\otimes A, W)\oplus \Hom(\mathcal{A}^{1}V)\oplus \Hom(M,W)\stackrel{D_2}{\longrightarrow}\\
  & \Hom(\otimes^3 \g, W)\oplus \Hom(\mathcal{A}^{2}, V)\oplus \Hom(\mathcal{A}^{1},W)\stackrel{D_3}{\longrightarrow}\cdots
\end{split}
\end{equation}
Thus, one can write the cochain complex in components as follows:
$$
\xymatrix{
  W \ar[d]_{-\delta_1}\ar[dr]^{\varphi\sharp}& & & & \\
  \Hom(\g,W)\qquad\ar[d]_{-\delta_1}\ar@<.5em>[dr]^ {-h_1}\ar@<1em>[dr]_{-h_2} \ar@<.5em>[drr]^ {\qquad-l_1} \ar@<.1em>[drr]_{\quad\qquad\qquad-l_2} &\hspace{-1cm} \oplus  \Hom(M,V)  \ar[d]_{\delta_2}\ar[dr]^{\varphi^\sharp}&&& \\
\Hom(A\otimes A, W)\qquad\ar[d]_{-\delta_1}\ar@<.5em>[dr]^ {-h_1}\ar@<1em>[dr]_{-h_2}
\ar@<.5em>[drr]^ {\qquad-l_1} \ar@<.1em>[drr]_{\quad\qquad\qquad-l_2} & \oplus \Hom(\mathcal{A}^1, V)
            \ar[d]_{\delta_2}\ar[dr]^{\varphi^\sharp}    & \hspace{-.5cm}\oplus\Hom(M,W)\ar[d]_{-\delta_3}  &&\\
\Hom(\otimes^3 \g, W) \qquad\ar[d]_{-\delta_1}\ar@<.5em>[dr]^ {-h_1}\ar@<1em>[dr]_{-h_2}
\ar@<.5em>[drr]^ {\qquad-l_1} \ar@<.1em>[drr]_{\quad\qquad\qquad-l_2} &\oplus \Hom(\mathcal{A}^2, V)
            \ar[d]_{\delta_2}  \ar[dr]^{\varphi^\sharp}   & \hspace{-.5cm} \oplus \Hom(\mathcal{A}^1,W)\ar[d]_{-\delta_3} & &  \\
\Hom(\otimes^4 \g, W)\qquad    & \oplus  \Hom(\mathcal{A}^3, V)  \    &\oplus \Hom(\mathcal{A}^2,W)& &  \\
        }
$$

$\vdots$

For 1-cochain $(N_0,N_1)\in \Hom(\g,W)\oplus\Hom(M,V)$ such that
\begin{eqnarray}
\label{NN01}
&&N_0\al(x)=\alw N_0(x),\quad \alv N_1(m)=N_1\alm(m),
\end{eqnarray}
 the coboundary map is
%%%%%%%%%%%%%%%%%%%%%%%%%%%
\begin{eqnarray*}
D_1(N_0,N_1)(x,y)&=&-\delta_1N_0(x,y)\notag\\
&=&N_0(x)\cdot y+ x\cdot N_0(y) - N_0(xy),\\
D_1(N_0,N_1)(x,m)&=&(-h_1N_0+\delta_2 N_1)(x,m)\notag\\
&=&N_0(x)\trr m+ x\cdot N_1(m) - N_1(x\cdot m),\\
D_1(N_0,N_1)(m,x)&=&(-h_2N_0+\delta_2 N_1)(m,x)\notag\\
&=&N_1(m)\cdot x+ m\trl N_0(x) - N_1(m\cdot x),\\
D_1(N_0,N_1)(m)&=&(-l_1N_0+\varphi^\sharp  N_1)(m)\notag\\
&=&\varphi\circ N_1(m)-N_0\circ f(m).
\end{eqnarray*}
Thus a 1-cocycle is   $(N_0,N_1)\in \Hom(\g,W)\oplus\Hom(M,V)$, such that
%%%%%%%%%%%%%%%%%%%%%%%%%%%
\begin{eqnarray}
%\label{N0N1}
%&&N_0\al(x)=\alw N_0(x),\quad \alv N_1(m)=N_1\alm(m),\\
&&\varphi\circ N_1(m)-N_0\circ f(m)=0,\\
&&N_0(x)y + x N_0(y) - N_0(xy)=0,\\
&&N_0(x) \trr m + x\cdot N_1(m) - N_1(x\cdot m)=0,\\
&&N_1(m) \cdot x + m\trl N_0(x) - N_1(m\cdot x)=0.
\end{eqnarray}

For a 2-cochain
$$(\omega,\mu,\nu,\theta)\in \Hom(A\otimes A, W)\oplus \Hom(\mathcal{A}^1,V)\oplus \Hom(M,W),$$
 we get
\begin{eqnarray*}
D_2(\omega,\mu,\nu,\theta)\in &&\Hom(\otimes^3 \g, W)\oplus \Hom(\mathcal{A}^2,V)\oplus \Hom(\mathcal{A}^1,W),
\end{eqnarray*}
where
\begin{eqnarray*}
 &&\Hom(\mathcal{A}^2,V):=\Hom(A\otimes A\otimes M,V)\oplus \Hom(A\otimes M\otimes A,V)\oplus \Hom(M\otimes A\otimes A,V),
%&& \Hom(\mathcal{A}^1,W):= \Hom(A\otimes M,W)\oplus \Hom(M\otimes A,W),
\end{eqnarray*}
and $D_2$ is defined by
\begin{eqnarray*}
D_2(\omega,\mu,\nu,\theta)(x, y, z)&=&-\delta_1\omega(x, y,  z)\notag\\
&=&\alin x\cdot \omega (y, z)+\omega (\al(x), yz)\\
&&-\omega(xy,\al(z))-\omega (x, y)\cdot \al(z),\\
D_2(\omega,\mu,\nu,\theta)(x, y,m)&=&(-h_1\omega+\delta_2 \mu)(x, y,m)\notag\\
&=&\al(x)\cdot\mu(y,m)-\omega(x,y)\trr\alm(m)\nonumber\\
&&+\mu(\al(x),y\cdot m)-\mu(xy,\alm(m)),\\
D_2(\omega,\mu,\nu,\theta)(x, m, y)&=&\delta_2 (\mu+\nu)(x, y,m)\notag\\
&=&\al(x)\cdot \nu(m,y))+\mu(\al(x), m\cdot y)\\
&&-\nu(x\cdot m, \al(y))-\mu(x,m)\cdot\al(y),\\
D_2(\omega,\mu,\nu,\theta)(m,x, y)&=&(-h_2\omega+\delta_2 \nu)(x, y,m)\notag\\
&=&\alm(m)\trl\omega(x,y)+\nu(\alm(m), xy)\\
&&-\nu(m,x)\cdot \al(y)-\nu(m\cdot x,\al(y)),\\
D_2(\omega,\mu,\nu,\theta)(x,m)&=&(-l_1\omega+\varphi^\sharp\mu-\delta_3\theta)(x,m)\notag\\
&=&\theta ( x\cdot m )+\varphi(\mu (x,m))-\omega (x, f(m))-x\cdot\theta (m),\\
D_2(\omega,\mu,\nu,\theta)(m,x)&=&(-l_2\omega+\varphi^\sharp\nu-\delta_3\theta)(m,x)\notag\\
&=&\theta ( m\cdot x )+\varphi(\nu (m,x))-\omega (f(m),x)-\theta (m)\cdot x.
\end{eqnarray*}
Thus a 2-coboundary is  $(\omega,\mu,\nu,\theta)\in \Hom(A\otimes A, W)\oplus \Hom(\g,\Hom(M,V))\oplus \Hom(M,W)$ such that
$(\omega,\mu,\nu,\theta)=D_1(N_0,N_1),$ i.e.
\begin{eqnarray*}
&&\theta(m)=\varphi\circ N_1(m)-N_0\circ f(m),\\
&&\omega(x,y)=N_0(x)\cdot y + x\cdot N_0(y) - N_0(xy),\\
&&\mu(x,m)=N_0(x) \trr m + x\cdot N_1(m) - N_1(x\cdot m),\\
&&\nu(m,x)=N_1(m) \cdot x + m\trl N_0(x) - N_1(m\cdot x),
\end{eqnarray*}
and a 2-cocycle is  $(\omega,\mu,\nu,\theta)\in \Hom(A\otimes A, W)\oplus \Hom(\g,\Hom(M,V))\oplus \Hom(M,W)$ such that
\begin{eqnarray}
\label{2-coc01}
&&\theta ( x\cdot m )+\varphi (\mu (x,m))-x\cdot\theta (m)-\omega (x, f(m))=0,\\
\label{2-coc02}
&&\theta ( m\cdot x )+\varphi(\nu (m,x))-\omega (f(m),x)-\theta (m)\cdot x=0,\\
\label{2-coc03}
&&\al(x)\cdot \omega (y, z)+\omega (\al(x), yz)-\omega(xy,\alin z)-\omega (x, y)\cdot \al(z)=0,\\
\label{2-coc04}
&&\al(x)\cdot\mu(y,m)+\mu(\al(x),y\cdot m)-\omega(x,y)\trr\alm(m)-\mu(xy,\alm(m))=0,\\
\label{2-coc05}
&&\alm(m)\trl\omega(x,y)+\nu(\alm(m), xy)-\nu(m,x)\cdot \al(y)-\nu(m\cdot x,\al(y))=0,\\
\label{2-coc06}
&&\al(x)\cdot \nu(m,y))+\mu(\al(x), m\cdot y)-\nu(x\cdot m, \al(y))-\mu(x,m)\cdot\al(y)=0.
\end{eqnarray}

\begin{pro}\label{prop:1cob}
The space of 2-coboundaries $B^2((M,\g, f),(V,W,\varphi))$ is contained in the space of 2-cocycles $Z^2((M,\g, f),(V,W,\varphi))$, i.e., $D^2=D_2\circ D_1=0$.
\end{pro}

\begin{proof}
We should verify that if $(\omega,\mu,\nu,\theta)=D_1(N_0,N_1)$, then  $D_2(\omega,\mu,\nu,\theta)=D_2D_1(N_0,N_1)=0$. Thus it must satisfy \eqref{2-coc01}--\eqref{2-coc06}.
In fact,
\begin{eqnarray*}
&&\theta ( x\cdot m )+\varphi (\mu (x,m))-x\cdot\theta (m)-\omega (x, f(m))\\
&=&\varphi\circ N_1( x\cdot m)-N_0\circ f( x\cdot m)\\
&&+\varphi(N_0(x)\trr m + x\cdot N_1(m) - N_1(x\cdot m))\\
&&-x\cdot(\varphi\circ N_1(m)-N_0\circ f(m))\\
&&-N_0(x)\cdot f(m) - x\cdot N_0(f(m)) +N_0(xf(m))\\
&=&\varphi\circ N_1( x\cdot m)-N_0(xf(m))\\
&&+\varphi(N_0(x) \trr m) + x\cdot\varphi( N_1(m)) - \varphi\circ N_1(x\cdot m)\\
&&-x\cdot \varphi (N_1(m))+x\cdot N_0(f(m))\\
&&-N_0(x)\cdot f(m) - x\cdot N_0(f(m)) +N_0(xf(m))\\
&=&0.
\end{eqnarray*}
where $\varphi(N_0(x) \trr m)=N_0(x)\cdot f(m)$ by  \eqref{deflm02}.
\begin{eqnarray*}
&&\al(x)\cdot\mu(y,m)+\mu(\al(x),y\cdot m)-\omega(x,y)\trr\alm(m)-\mu(xy,\alm(m))\\
&=&\al(x)\cdot\big(N_0(y) \trr m + y\cdot N_1(m) - N_1(y\cdot m)\big)\\
&&+N_0(\al(x)) \trr (y\cdot m) + \al(x)\cdot N_1(y\cdot m) - N_1(\al(x)\cdot (y\cdot m))\\
&&-(N_0(x)\cdot y + x\cdot N_0(y) - N_0(xy))\trr\alm(m)\\
&&-N_0(xy) \trr \alm(m) + (xy)\cdot N_1(\alm(m)) - N_1( (xy)\cdot \alm(m))\\
&=&\al(x)\cdot(N_0(y) \trr m) + \al(x)\cdot(y\cdot N_1(m)) - \al(x)\cdot N_1(y\cdot m)\\
&&+N_0(\al(x)) \trr (y\cdot m) + \al(x)\cdot N_1(y\cdot m) - N_1(\al(x)\cdot (y\cdot m))\\
&&-(N_0(x)\cdot y)\trr\alm(m) - (x\cdot N_0(y))\trr\alm(m) + N_0(xy)\trr\alm(m)\\
&&-N_0(xy) \trr \alm(m) - (xy)\cdot N_1(\alm(m))+ N_1( (xy)\cdot \alm(m))\\
&=&0.
\end{eqnarray*}
where $$\al(x)\cdot(N_0(y) \trr m)=(x\cdot N_0(y))\trr\alm(m),$$
$$\al(x)\cdot(y\cdot N_1(m))= (xy)\cdot \alv N_1(m)= (xy)\cdot N_1(\alm(m)),$$
$$N_0(\al(x)) \trr (y\cdot m) =(\alw N_0(x)\cdot y)\trr m=(N_0(x)\cdot y)\trr\alm(m)$$
by \eqref{mod01}, \eqref{deflm12} and  \eqref{NN01} respectively.

Therefore we have verified that \eqref{2-coc01} and \eqref{2-coc04} hold.
The other equalities can be checked similarly. This completes the proof.
\end{proof}

\begin{defi}\label{def:2group}
The second cohomology group  of  $(M,\g, f)$ with coefficients in $(V,W,\varphi)$ is defined as the quotient of
the space of 2-coboundaries module the space of 2-cocycles, which is denoted by $\mathbf{H}^2((M,\g, f),(V,W,\varphi))$.
\end{defi}

\section{Applications}
\subsection{Infinitesimal deformations}

Let $(M,\g, f)$ be a Hom-associative algebra in $\LM$ and $\theta: M \to\g,~\omega:A\otimes A\to \g$, $\mu:\g \otimes M \to M, ~\nu: M\otimes A\to M$ be linear maps. Consider a $\lambda$-parametrized family of linear operations:
\begin{eqnarray*}
 \dM\dlam (m)&\triangleq&\dM (m)+\lambda\theta (m),\\
 {x\cdlam y}&\triangleq& xy+ \lambda\omega (x, y),\\
 {x\cdot_\lam m}&\triangleq&  x\cdot m + \lambda\mu(x, m),\\
  {m\cdot_\lam x}&\triangleq&  m\cdot x + \lambda\nu(m, x).
 \end{eqnarray*}
%\comment{Since we use $[\cdot,\cdot]$ at the beginning, here we need to unify.}

If  $(M_\lambda, \g\dlam, f\dlam)$ forms a  Hom-associative algebra in $\LM$, then we say that
$(\omega,\mu,\nu,\theta)$ generates a 1-parameter infinitesimal deformation of $(M,\g, f)$.

\begin{thm}\label{thm:deformation} With the notations above,
$(\omega,\mu,\nu,\theta)$ generates a $1$-parameter infinitesimal deformation of $(M,\g, f)$  if and only if the following conditions hold:
\begin{itemize}
  \item[\rm(i)] $(\omega,\mu,\nu,\theta)$ is a $2$-cocycle of $(M,\g, f)$ with coefficients in the adjoint representation;

  \item[\rm(ii)] $(M,\g,\theta)$ is a Hom-associative algebra structure in $\LM$ with multiplication $\omega$ on $\g$ and bimodule structures of $\g$ on $M$ by $\mu, \nu$.
\end{itemize}
\end{thm}

\begin{proof}
If $(M_\lambda, \g\dlam, f\dlam)$  is a Hom-associative algebra in $\LM$, then by Definition \ref{defi:Lie 2}, $f\dlam$ is a bimodule map.
Thus we have
\begin{eqnarray*}
&&\dM\dlam (x\cdlam m)-x\cdlam \dM\dlam (m)\\
&=&(\dM+\lambda\theta )( x\cdot m + \lambda\mu(x, m))-x(\dM (m)+\lambda\theta (m))-\lambda\omega (x, \dM (m)+\lambda\theta (m))\\
&=&\dM(x\cdot m)+\lambda\Big(\theta( x\cdot m )+\dM\mu(x,m)\Big)+\lambda^2\theta \omega (x,m)\\
&&-x f(m)-\lambda\Big(\omega(x, f(m)+x\theta (m)\Big)-\lambda^2\mu(x,\theta (m))\\
&=&0,
\end{eqnarray*}
%\comment{You see you made many mistake in $\omega $ and $\mu$, here and through the paper}
which implies that
\begin{eqnarray}
\label{eq:2-coc01}\theta ( x\cdot m )-x\theta (m)+\dM \mu (x,m)-\omega (x, \dM (m))&=&0,\\
\label{eq:2-coc02}\theta \mu(x,m)-\omega (x,\theta (m))&=&0.
\end{eqnarray}
Similar computation shows that
\begin{eqnarray}
\label{eq:2-coc03}\theta (m\cdot x)-\theta (m)x+\dM \nu (m,x)-\omega (\dM (m), x)&=&0,\\
\label{eq:2-coc04}\theta \nu(m,x)-\omega (\theta (m),x)&=&0.
\end{eqnarray}

Since $\g\dlam$ is a Hom-associative algebra, we have
\begin{eqnarray*}
&&(x\cdlam y)\cdlam \al(z)-\al(x)\cdlam(y\cdlam z)\\
&=&(xy+\lambda\omega (x, y))\cdlam  \al(z)-\al(x)\cdlam(yz+\lam\omega(y, z))\\
&=&(xy)\al(z)-\al(x)(yz)+\lambda\Big(\omega (xy,\al(z))+\omega (x, y)\al(z)\Big)\\
&&+\lambda^2\Big(\omega(\omega (x,y),\al(z))-\omega(\al(x),\omega (y,z))\Big)\\
&=&0,
\end{eqnarray*}
which implies that
\begin{eqnarray}
\label{eq:2-cocycle1'}\omega (xy,\al(z))+\omega (x, y)\al(z)-\omega (\al(x), yz)-\al(x)\omega (y,z)&=&0,\\
\label{eq:2-cocycle1''}\omega (\omega (x,y),\al(z))-\omega(\al(x),\omega (y,z))&=&0.
\end{eqnarray}

Since $M_\lambda$ is a left module of $\g\dlam$, we have
\begin{eqnarray*}
&&(x\cdlam y)\cdlam\alm(m)-\al(x)\cdlam(y\cdlam m)\\
%&=&[xy+\lambda\omega (x, y), \alm(m)]\dlam+c.p.\\
%&=&[xy,\alm(m)]\dlam+c.p.+\lambda[\omega (x, y),\alm(m)]\dlam+c.p.\\
&=&(xy)\cdot\alm(m)-\al(x)\cdot (y\cdot m) \\
\nonumber&&+\lambda\Big( \mu(xy,\alm(m))+\omega(x,y)\cdot \alm(m)\\
&&-\al(x)\mu(y,m)-\mu(\al(x), y\cdot m)\Big)\\
&&+\lambda^2\Big(\mu(\omega (x,y),\alm(m))-\mu(\al(x), \mu(y, m))\Big)\\
&=&0,
\end{eqnarray*}
which implies that
\begin{eqnarray}
\nonumber&&\mu(xy,\alm(m))+\omega(x,y)\cdot \alm(m)\\
\label{eq:2-coc31}
&&-\al(x)\mu(y,m)-\mu(\al(x), y\cdot m)=0,\\
\label{eq:2-coc32}
&&\mu(\omega (x,y),\alm(m))-\mu(\al(x), \mu(y, m))=0.
\end{eqnarray}

Similar computation shows that
\begin{eqnarray}
\nonumber&&\nu(\alm(m), xy)+ \alm(m)\cdot\omega(x,y)\\
\label{eq:2-coc33}
&&-\nu(m, x)\al(y)-\nu(m\cdot x, \al(y))=0,\\
\label{eq:2-coc34}
&&\nu(\alm(m),\omega (x,y))-\nu(\nu(m,x), \al(y))=0.
\end{eqnarray}

By \eqref{eq:2-coc01},  \eqref{eq:2-coc31} and \eqref{eq:2-coc33}, we find that $(\omega,\mu,\nu,\theta)$ is a 2-cocycle of $(M,\g, f)$ with the coefficients in the adjoint representation.
Furthermore, by \eqref{eq:2-coc02}, \eqref{eq:2-coc32} and \eqref{eq:2-coc34}, $(M,\g, \theta)$ with multiplication $\omega$ is a Hom-associative algebra in $\LM$.
\end{proof}

\medskip

Now we introduce the notion of Nijenhuis operators which gives trivial deformations.

Let $(M,\g, f)$ be a Hom-associative algebra in $\LM$ and $N=(N_0,N_1)$ be a pair of linear maps $N_0:\frkg \to \frkg $ and $N_1: M\to  M $
such that $\dM\circ N_1=N_0\circ \dM$. Define an exact 2-cochain
 $$(\omega,\mu,\nu,\theta)=D(N_0,N_1)$$
 by differential $D$ discussed above, i.e.,
  \begin{eqnarray*}
   \theta(m)&=&\dM\circ N_1(m)-N_0\circ \dM(m),\\
\omega(x, y)=x\cdot_Ny&=&N_0(x)y + x N_0(y) - N_0(xy),\\
\mu(x,m)=x\cdot_N m&=&N_0(x) m + x\cdot N_1(m) - N_1(x\cdot m),\\
\nu(m,x)=m\cdot_N x&=&N_1(m) x + m\cdot N_0(x) - N_1(m\cdot x).
 \end{eqnarray*}

 \begin{defi}\label{defi:Nijenhuis}
  A pair of linear maps $N=(N_0,N_1)$ is called a Nijenhuis operator if for all $x,y\in\g $ and $m\in M $, the following conditions are satisfied:
   \begin{itemize}
     \item[(i)] $\Imm (\dM\circ N_1-N_0\circ \dM)\in \Ker N_0,$
     \item[(ii)] $N_0(x\cdot_N y)=N_0(x)N_0(y),$
     \item[(iii)] $N_1(x\cdot_N m)=N_0(x)\cdot N_1(m),$
     \item[(iv)] $N_1(m\cdot_N x)=N_1(m)\cdot N_0(x).$
   \end{itemize}
 \end{defi}

\begin{defi}
A deformation is said to be trivial if there exists a pair of  linear maps $N_0:\frkg \to \frkg, ~N_1:  M \to  M $,
such that $(T_0,T_1)$ is a morphism from $(M\dlam,\g\dlam,\dM\dlam)$ to $(M,\g,\dM) $, where $T_0 = \id + \lambda N_0$,
$T_1 = \id + \lambda N_1$.
\end{defi}

Note that $(T_0,T_1)$ is a morphism means that
\begin{eqnarray}
\label{eq:trivialdeform0} \dM \circ T_1(m)&=&T_0\circ \dM\dlam(m),\\
\label{eq:trivialdeform1} T_0(x\cdot\dlam y)&=&T_0(x)\, T_0(y),\\
\label{eq:trivialdeform2} T_1(x\cdot\dlam m)&=&T_0(x)\cdot T_1(m),\\
\label{eq:trivialdeform3} T_1(m\cdot\dlam x)&=&T_1(m)\cdot T_0(x).
\end{eqnarray}
Now we consider what conditions that $N=(N_0,N_1)$ should satisfy.
From \eqref{eq:trivialdeform0}, we have
$$\dM (m)+\lambda \dM N_1(m)=(\id + \lambda N_0)(\dM (m)+\lambda\theta (m))=\dM (m)+\lambda N_0(\dM (m))+\lambda \theta (m)+\lambda^2N_0\theta (m).$$
Thus, we have
$$\theta (m)=(\dM N_1-N_0\dM) (m),$$
$$N_0\theta (m)=0.$$
It follows that $N$ must satisfy the following condition:
\begin{align}\label{Nijenhuis0}
N_0(\dM N_1-N_0\dM)(m)=0.
\end{align}

For \eqref{eq:trivialdeform1}, the left hand side is equal to
$$xy+\lambda N_0(xy)+\lambda\omega (x,y)+\lambda^2 N_0\omega (x, y),$$
and the right hand side is equal to \begin{eqnarray*}
xy+\lambda N_0(x) y+\lambda x N_0(y)+\lambda^2 N_0(x)\, N_0(y).
\end{eqnarray*}
Thus, \eqref{eq:trivialdeform1} is equivalent to
$$\omega (x, y) =N_0(x) y + x N_0(y) - N_0(xy),$$
$$N_0\omega (x, y) = N_0(x) N_0(y). $$
It follows that $N$ must satisfy the following condition:
\begin{align}\label{Nijenhuis1}
N_0(x) N_0(y) - N_0(N_0(x) y) - N_0(x N_0(y)) + N_0^2(xy)=0.
\end{align}

For \eqref{eq:trivialdeform2},  the left hand side is equal to
$$x\cdot m+\lambda \mu(x,m) +\lambda N_1(x\cdot m)+\lambda^2 N_1\mu(x,m),$$
and the right hand side is equal to
$$x\cdot m+\lambda N_0(x)\cdot m+\lambda x\cdot N_1(m)+\lambda^2 N_0(x)\cdot N_1(m).$$
Thus, \eqref{eq:trivialdeform2} is equivalent to
$$\nu(x, m) = N_0(x)\cdot  m + x\cdot N_1(m) - N_1(x\cdot m),$$
$$N_1\nu(x, m) = N_0(x)\cdot N_1(m)+N_2(x,\theta (m)). $$
It follows that $N$ must satisfy the following condition:
\begin{align}\label{Nijenhuis2}
N_0(x)\cdot N_1(m)- N_1(N_0(x)\cdot m) - N_1(x\cdot N_1(m)) + N_1^2(x\cdot m)=0.
\end{align}

Similarly, form  \eqref{eq:trivialdeform3} we obtain
\begin{align}\label{Nijenhuis3}
N_1(m)\cdot N_0(x)- N_1(m\cdot N_0(x)) - N_1(N_1(m)\cdot x) + N_1^2(m\cdot x)=0.
\end{align}

A Nijenhuis operator $(N_0,N_1)$ could give a trivial deformation by setting
\begin{equation}\label{eq:omega exact}
(\omega,\mu,\nu,\theta)=D(N_0,N_1).
\end{equation}

\begin{thm}\label{thm:Nijenhuis}
Let $N=(N_0,N_1)$ be a Nijenhuis operator. Then a deformation
can be obtained by putting
\begin{equation}\label{Nijenhuis}
\left\{\begin{array}{rll}
\theta (m) &=&(\dM N_1-N_0\dM) (m),\\
\omega (x, y)&=&N_0(x) y + x N_0(y) - N_0(x)y,\\
\mu(x,m)&=&N_0(x)\cdot m + x\cdot N_1(m) - N_1(x\cdot m)\\
\nu(m,x)&=&m\cdot N_0(x) +  N_1(m)\cdot x - N_1(m\cdot x).
\end{array}\right.
\end{equation}
Furthermore, this deformation is trivial.
\end{thm}

\begin{proof} Since $(\omega,\mu,\nu,\theta)=D(N_0,N_1)$, it is obvious that $(\omega,\mu,\nu,\theta)$ is a 2-cocycle.
It is easy to check that $(\omega,\mu,\nu,\theta)$ defines a Hom-associative algebra $(M,\g,\theta)$ in $\LM$ structure.
Thus by Theorem \ref{thm:deformation}, $(\omega,\mu,\nu,\theta)$ generates a deformation.
\end{proof}

\emptycomment{
\begin{eqnarray*}
\theta \mu(x,m)&=&(\dM N_1-N_0\dM)([N_0(x), m] + [x, N_1(m)] - N_1(x\cdot m))\\
&=&\dM N_1([N_0(x), m] + [x, N_1(m)] - N_1(x\cdot m))-N_0\dM([N_0(x), m] + [x, N_1(m)] - N_1(x\cdot m))\\
&=&\dM N_1([N_0(x), m] + \dM N_1[x, N_1(m)] - \dM N_1 N_1(x\cdot m))\\
&&-N_0[N_0(x), \dM(m)] - N_0[x, \dM N_1(m)] + N_0 \dM N_1(x\cdot m))\\
\end{eqnarray*}
\begin{eqnarray*}
\omega (x,\theta (m))&=&[N_0(x), (\dM N_1-N_0\dM) (m)] + [x, N_0(\dM N_1-N_0\dM) (m)] - N_0[x, (\dM N_1-N_0\dM) (m)]\\
&=&[N_0(x), \dM N_1 (m)]-[N_0(x),N_0\dM (m)] - N_0[x, \dM N_1 (m)]+ N_0[x, N_0\dM (m)]\\
\end{eqnarray*}
}

%\subsection{Formal deformations}
Now we consider the general formal deformations.
Let $(M,\g, f)$ be a Hom-associative algebra in $\LM$ and $\theta_i: M \to\g, ~\omega_i:A\otimes A\to \g , ~\mu_i:\g \otimes M \to M, ~\nu_i: M \otimes A \to M, i\geqslant 0$ be linear maps where $\theta_0=f,  ~\omega_0(x,y)=xy,  ~\mu_0(x,m)=x\cdot m,  ~\nu_0(m,x)=m\cdot x$.
Consider a $\lambda$-parametrized family of linear operations:
\begin{eqnarray*}
\dM\dlam (m)&\triangleq&\dM (m)+\lambda\theta_1 (m)+\lambda^2\theta_1 (m)+\cdots,\\
\omega\dlam(x,y)={x\cdlam y}&\triangleq& xy+\lambda\omega_1 (x, y)+\lambda^2\omega_2 (x, y)+\cdots,\\
\mu\dlam(x,m)={x\cdot\dlam m}&\triangleq& x\cdot m+ \lambda\mu_1(x, m)+ \lambda^2\mu_2(x, m)+\cdots,\\
\nu\dlam(m,x)={m\cdot\dlam x}&\triangleq&m\cdot x+ \lambda\nu_1(m,x)+ \lambda^2\nu_2(m,x)+\cdots.
 \end{eqnarray*}
In order that  $(M\dlam, \g\dlam, f\dlam)$ forms a Hom-associative algebra in $\LM$, we must have
\begin{eqnarray}
&&\label{eq:formal00}f\dlam\nu\dlam (x,m)=\omega\dlam(x,f\dlam(m)),\\
&&\label{eq:formal01}\omega\dlam (\omega\dlam(x,y), \al(z))=\omega(\al(x), \omega\dlam(y,z)),\\
&&\label{eq:formal02}\mu\dlam (\omega\dlam(x,y), \alm(m))=\mu\dlam (\al(x), \mu\dlam(y,m)),\\
&&\label{eq:formal03}\nu\dlam (\alm(m),\omega\dlam(x,y))=\nu\dlam (\nu\dlam(m,x), \alm(y)),\\
&&\label{eq:formal04}\mu\dlam (\al(x), \nu\dlam(m,y))=\nu\dlam (\mu\dlam(x,m), \al(y)).
\end{eqnarray}
which implies that
\begin{eqnarray}
&&\label{eq:formal0}\sum_{i+j=k}f_i\nu_j(x,m)=\omega_i(x,f_j(m)),\\
&&\label{eq:formal1}\sum_{i+j=k}\omega_i (\omega_j(x,y),\al(z))=\sum_{i+j=k}\omega_i (\al(x),\omega_j(y,z)),\\
&&\label{eq:formal2}\sum_{i+j=k}\mu_i (\omega_j(x,y), \alm(m))=\sum_{i+j=k}\mu_i (\al(x), \mu_j(y,m)),\\
&&\label{eq:formal2}\sum_{i+j=k}\nu_i (\alm(m), \omega_j(x,y))=\sum_{i+j=k}\nu_i (\nu_j(m,x), \al(y)),\\
&&\label{eq:formal3}\sum_{i+j=k}\mu_i (\al(x), \nu_j(m,y))=\sum_{i+j=k}\nu_i (\mu_j(x,m), \al(y)).
\end{eqnarray}

For $k=0$, conditions \eqref{eq:formal00}--\eqref{eq:formal04} are equivalent to $(\omega_0,\mu_0, \nu_0,\theta_0)$ is a Hom-associative algebra in $\LM$.

For $k=1$, these conditions  are equivalent to
\begin{eqnarray}
&&\theta (x\cdot m)+f \nu_1 (x,m)=\omega_1 (x, f(m))+x\theta (m),\\
&&\omega_1(xy, \al(z))+\omega_1 (x, y)\al(z)=\omega_1(\al(x), yz)+\al(x)\omega_1 (y,z),\\
&&\mu_1(xy, \alm(m))+\omega_1 (x, y) \cdot\alm(m)=\al(x)\mu_1(y,m)+\mu_1(\al(x), y\cdot m),\\
&&\nu_1(\alm(m), xy)+ \alm(m)\cdot\omega_1(x,y)=\nu_1(m, x)\al(y)+\nu_1(m\cdot x, \al(y)),\\
&&\mu_1(\al(x), m\cdot y)+\al(x)\cdot \nu_1(m,y))=\nu_1(x\cdot m, \al(y))+\mu_1(x,m)\cdot\al(y).
\end{eqnarray}
Thus $(\omega_1,\mu_1,\nu_1,\theta_1)\in C^{2}((M,\g, f),(M,\g, f))$ is a 2-cocycle.

\begin{defi}
The 2-cochain $(\omega_1,\mu_1,\nu_1,\theta_1)$ is called the infinitesimal of $(\omega\dlam,\mu\dlam,\nu\dlam, f\dlam)$. More generally, if $(\omega_i,\mu_i, \nu_i, f_i)=0$ for $1\leqslant i\leqslant (n -1)$,
and $(\omega_n,\mu_n, \nu_n, f_n)$ is a non-zero cochain in $C^{2}((M,\g, f),(M,\g, f))$, then $(\omega_n,\mu_n,\nu_n,f_n)$ is called the $n$-infinitesimal of
the deformation $(\omega\dlam,\mu\dlam,\nu\dlam,  f\dlam)$.
\end{defi}

Let $(\omega\dlam,\mu\dlam,\nu\dlam,f\dlam)$ and $(\omega'\dlam,\mu'\dlam,\nu'\dlam,f'\dlam)$ be two deformation.
We say that they are equivalent if there exists a formal isomorphism
$(\Phi\dlam,\Psi\dlam):(M'\dlam, \g'\dlam, f'\dlam)\to (M\dlam, \g\dlam, f\dlam)$ such that
$\omega'\dlam(x,y)=\Psi^{-1}\dlam\omega\dlam(\Psi\dlam(x),\Psi\dlam(y))$.

A deformation $(\omega\dlam,\mu\dlam,\nu\dlam,f\dlam)$ is said to be the trivial deformation if it is  equivalent to $(\omega_0,\mu_0,\nu_0, \theta_0)$.

\begin{thm}\label{thm-deform}
Let $(\omega\dlam,\mu\dlam,\nu\dlam,f\dlam)$ and $(\omega'\dlam,\mu'\dlam,\nu'\dlam, f'\dlam)$ be equivalent deformations of  $(M,\g, f)$,
then the first-order terms of them belong to the same cohomology class in the second cohomology group $H^2((M,\g, f),(M,\g, f))$.
\end{thm}

\begin{proof}
Let $(\Phi\dlam,\Psi\dlam):(M\dlam, \g\dlam, f\dlam)\to (M'\dlam, \g'\dlam, f'\dlam)$ be an equivalence
where $\Phi\dlam=\id_M+\lambda\phi_1+\lambda^2\phi_2+\cdots$ and $\Psi\dlam=\id_M+\lambda\psi_1+\lambda^2\psi_2+\cdots$.
Then we have
$\Psi\dlam\omega'\dlam(x,y)=\omega\dlam(\Psi\dlam(x),\Psi\dlam(y))$,
$\Psi\dlam\nu'\dlam(x,m)=\nu\dlam(\Phi\dlam(x),\Psi\dlam(m))$.
Then by expanding the above equality, we have $(\theta_1,\omega_1,\mu_1,\nu_1)-(\theta'_1,\omega'_1,\mu'_1,\nu'_1)=D(\phi_1,\psi_1)$. Thus $(\theta_1,\omega_1,\mu_1,\nu_1)$ and $(\theta'_1,\omega'_1,\mu'_1,\nu'_1)$ are belong to the same cohomology class in the second cohomology group. The proof is finished.
\end{proof}

A Hom-associative algebra $(M,\g, f)$ in $\LM$ is called rigid if every deformation $(\omega\dlam,\mu\dlam,\nu\dlam,f\dlam)$ is equivalent to the trivial deformation.

\begin{thm}
If $H^2((M,\g, f),(M,\g, f))=0$, then $(M,\g, f)$ is rigid.
\end{thm}

\begin{proof}
Let $(\omega\dlam,\mu\dlam,\nu\dlam,f\dlam)$ be a  deformation of $(M,\g, f)$.
It follows from above Theorem \ref{thm-deform} that  $D(\omega\dlam,\mu\dlam,\nu\dlam,f\dlam)=0$, that is $(\omega\dlam,\mu\dlam,\nu\dlam,f\dlam)\in Z^2((M,\g, f),(M,\g, f))$. Now assume $H^2((M,\g, f))=0$, we can find $(N_0,N_1)$ such that $(\omega\dlam,\mu\dlam,\nu\dlam,f\dlam)=D(N_0,N_1)$.
Thus $(\omega\dlam,\mu\dlam,\nu\dlam,f\dlam)$ is equivalent to the trivial deformation.
This proof is completed.
\end{proof}

\subsection{Abelian extensions}

\begin{defi}
Let $(M,\g, f)$ be a Hom-associative algebra in $\LM$. An extension of $(M,\g, f)$ is a
short exact sequence such that $\mathrm{Im}(i_0)=\mathrm{Ker}(p_0)$ and $\mathrm{Im}(i_1)=\mathrm{Ker}(p_1)$ in the following commutative diagram
\begin{equation}\label{eq:ext1}
\CD
  0 @>0>>  \frkh @>i_1>> \widehat{M} @>p_1>>  M  @>0>> 0 \\
  @V 0 VV @V \varphi VV @V \widehat{\dM} VV @V\dM VV @V0VV  \\
  0 @>0>> W @>i_0>> \widehat{\g} @>p_0>> \g @>0>>0
\endCD
\end{equation}
where $(V,W,\varphi)$ is a Hom-associative algebra in $\LM$.
\end{defi}

We call $(\widehat M,\widehat\g, \widehat f)$ an extension of $(M,\g, f)$ by
$(V,W,\varphi)$, and denote it by $\widehat{\E}$.
It is called an abelian extension if $(V,W,\varphi)$ is an abelian Hom-associative algebra in $\LM$ (this means that the multiplication on $W$ is zero, the bimodule of $W$ on $V$ is trivial).

A splitting $\sigma=(\sigma_0,\sigma_1):(M,\g, f)\to (\widehat M,\widehat \g, \widehat f)$ consists of linear maps
$\sigma_0:{A}\to\widehat{\g}$ and $\sigma_1:M\to \widehat M$
 such that  $p_0\circ\sigma_0=\id_{{A}}$, $p_1\circ\sigma_1=\id_{M}$ and $\widehat{f}\circ\sigma_1=\sigma_0\circ f$.

 Two extensions of Lie algebras
 $\widehat{\E}:0\longrightarrow(V,W,\varphi)\stackrel{i}{\longrightarrow}(\widehat M,\widehat \g, \widehat f)\stackrel{p}{\longrightarrow}(M,\g, f)\longrightarrow0$
 and  $\widetilde{\E}:0\longrightarrow(V,W,\varphi)\stackrel{j}{\longrightarrow}(\widetilde{M},\widetilde{\g}, \widetilde{f})\stackrel{q}{\longrightarrow}(M,\g, f)\longrightarrow0$ are equivalent,
 if there exists a morphism $F:(\widehat M,\widehat \g, \widehat f)\longrightarrow(\widetilde{M},\widetilde{\g}, \widetilde{f})$  such that $F\circ i=j$, $q\circ
 F=p$.

Let $(\widehat M,\widehat\g, \widehat f)$ be an extension of $(M,\g, f)$ by
$(V,W,\varphi)$ and $\sigma:(M,\g, f)\to (\widehat M,\widehat \g, \widehat f)$ be a splitting.
Define the following maps:
\begin{equation}\label{extension:bimod}
\left\{\begin{array}{rlclcrcl}
\cdot:&{A}\otimes V&\longrightarrow& V,&& x\cdot v&\triangleq&\sigma_0(x)\cdot v,\\
\cdot:&{A}\otimes W&\longrightarrow& W,&& x\cdot w&\triangleq&\sigma_0(x) w,\\
\trr:&W\otimes M&\longrightarrow&V,&&w\trr m&\triangleq&w\cdot\sigma_1(m),\\
\trl:&M\otimes W&\longrightarrow&V,&&m\trl w&\triangleq&\sigma_1(m)\cdot w,
\end{array}\right.
\end{equation}
for all $x\in {A}, v\in V, w\in W, m\in M$.

\begin{pro}\label{pro:2-modules}
With the above notations,  $(V,W,\varphi)$ is a bimodule of $(M,\g, f)$.
Furthermore, this bimodule structure does not depend on the choice of the splitting $\sigma$.
Moreover,  equivalent abelian extensions give the same  bimodule on $(V,W,\varphi)$.
\end{pro}

\begin{proof}  Firstly, we show that the bimodule is well-defined.
Since $\Ker p_{0} \cong W$, then for $w\in {W}$, we have $p_{0}(w)=0$.
By the fact that $(p_1,p_0)$ is a homomorphism between $(\widehat M,\widehat\g, \widehat f)$ and $(M,\g, f)$, we get
$$p_{1}(x\cdot w)=p_{1}(\sigma_0(x) w)=p_{0}\sigma_0(x)\cdot p_{0}(w)=p_{0}\sigma_0(x)\cdot 0=0.$$
Thus $x\cdot w\in  \ker p_{0} \cong {W}$.
Similar computations show that
$$p_{0}(x\cdot v)=p_{0}(\sigma_0(x)\cdot v)=p_{0}\sigma_0(x)\cdot p_{1}(v)=p_{0}\sigma_0(x)\cdot 0=0,$$
$$p_{0}(m\trl w)=p_{0}(\sigma_1(m)\cdot w)=p_{0}\sigma_1(m)\cdot p_{0}(w)=p_{0}\sigma_0(m)\cdot 0=0.$$
Thus $x\cdot v, m\trl w\in \Ker p_0=V$.

Now we  will show that these maps are independent of the choice of $\sigma$. In fact, if we choose another splitting
$\sigma':\g\to\wg$, then $p_0(\sigma_0(x)-\sigma'_0(x))=x-x=0$,
$p_1(\sigma_1(m)-\sigma'_1(m))=m-m=0$, i.e. $\sigma_0(x)-\sigma'_0(x)\in
\Ker p_0=W$, $\sigma_1(m)-\sigma'_1(m)\in \Ker p_1=V$. Thus,
$(\sigma_0(x)-\sigma'_0(x)(v+w)=0$, $(\sigma_1(m)-\sigma'_1(m))w=0$, which implies that the maps in \eqref{extension:bimod}  are independent
on the choice of $\sigma$. Therefore the bimodule structures are well-defined.

Secondly, we check that  $(V,W,\varphi)$ is indeed a bimodule of $(M,\g, f)$.
 Since $(V,W,\varphi)$ is an abelian Hom-associative algebra in $\LM$, we have
 \begin{eqnarray*}
   &&(xy)\cdot\alpha_V(v)-\alpha(x)\cdot(y\cdot v)\\
   &=&\sigma_0(xy)\cdot\alpha_V(v)-\sigma\alpha(x)\cdot(\sigma_0(y)\cdot v)\\
   &=&(\sigma_0(x)\sigma_0(y))\widehat{\alpha}(v)-\widehat{\alpha}\sigma_0(x)(\sigma_0(y)\cdot v)\\
   &=&0,
 \end{eqnarray*}
which implies that
$$(xy)\cdot\alpha_V(v)=\alpha(x)\cdot (y\cdot v).$$
Similarly, we have
$$(xy)\cdot\alpha_W(w)=\alpha(x)\cdot(y\cdot w).$$
Now $\varphi$ is a bimodule map since
\begin{eqnarray}
\varphi(x\cdot v)=\varphi(\sigma_0(x)\cdot v)=\sigma_0(x)\cdot\varphi(v)=x\cdot\varphi(v).
\end{eqnarray}
For $\trr: W\otimes M\to V$, we have
\begin{eqnarray}
\varphi( w\trr m)&=&\varphi(w\cdot\sigma_1(m))=w\cdot\widehat{f}\sigma_1(m)\notag\\
&=&w\cdot\sigma_0(f(m))=w\cdot f(m).
\end{eqnarray}
For $\trl: M\otimes W\to V$, we have
\begin{eqnarray}
\varphi( m\trl w)&=&\varphi(\sigma_1(m)\cdot w)=\widehat{f}\sigma_1(m)\cdot w\notag\\
&=&\sigma_0(f(m))\cdot w=f(m)\cdot w.
\end{eqnarray}
One check that these two maps  $\trr$ and $\trl$ satisfying  the conditions \eqref{deflm11}--\eqref{deflm16}.
Therefore, $(V,W,\varphi)$ is a bimodule of $(M,\g, f)$.

Finally, suppose that $\widehat{\E}$ and $\widetilde{\E}$ are equivalent abelian extensions,
and $F:(\widehat M,\widehat \g, \widehat f)\longrightarrow(\widetilde{M},\widetilde{\g}, \widetilde{f})$ be the morphism.
Choose linear sections $\sigma$ and $\sigma'$ of $p$ and $q$. Then we have $q_0F_0\sigma_0(x)=p_0\sigma_0(x)=x=q_0\sigma'_0(x)$,
thus $F_0\sigma_0(x)-\sigma'_0(x)\in \Ker q_0=W$. Therefore, we obtain
$$\sigma'_0(x)\cdot w=F_0\sigma_0(x)\cdot w=F_0(\sigma_0(x)\cdot w)=\sigma_0(x)\cdot w,$$
which implies that equivalent abelian extensions give the same module structures on $W$.
Similarly, we can show that equivalent abelian extensions also give the same $(V,W,\varphi)$.
Therefore, equivalent abelian extensions also give the same representation. The proof is finished.
\end{proof}

Let $\sigma:(M,\g, f)\to (\widehat M,\widehat \g, \widehat f)$  be a splitting of an abelian extension. Define the following linear maps:
\begin{equation}\label{extension:cocycle}
\left\{
\begin{array}{rlclcrcl}
\theta:& M &\longrightarrow& W ,&& \theta(m)&\triangleq&\widehat{\dM}\sigma_1(m)-\sigma_0(f(m)),\\
\omega:&\frkg \otimes\frkg&\longrightarrow& W ,&& \omega(x,y)&\triangleq&\sigma_0(x)\sigma_0(y)-\sigma_0(xy),\\
\mu:& A\otimes M&\longrightarrow&V,&& \mu(x,m)&\triangleq&\sigma_0(x)\cdot\sigma_1(m)-\sigma_1(x\cdot m),\\
\nu:&M\otimes A &\longrightarrow&V,&& \nu(m,x)&\triangleq&\sigma_1(m)\cdot\sigma_0(x)-\sigma_1(m\cdot x).
\end{array}\right.
\end{equation}
for all $x,y,z\in\frkg$, $m\in M $.

\begin{thm}\label{thm:2-cocylce}
With the above notations, $(\omega,\mu,\nu,\theta)$ is a $2$-cocycle of $(M,\g, f)$ with coefficients in $(V,W,\varphi)$.
\end{thm}

\begin{proof}
Since $\widehat f$ is a bimodule map, we have the equality
$$\widehat{\dM}\big(\sigma_0(x)\cdot\sigma_1(m))\big)=\sigma_0(x)\widehat{\dM}\sigma_1(m),$$
then by \eqref{extension:cocycle} we get that
\begin{eqnarray*}
\widehat{\dM}\big(\mu(x,m)+\sigma_1( x\cdot m )\big)=\sigma_0(x)\big(\sigma_0(f(m))+\theta(m)\big),
\end{eqnarray*}
\begin{eqnarray*}
\varphi(\mu(x,m))+\theta( x\cdot m )+\sigma_0(f( x\cdot m ))= \omega(x,f(m))+\sigma_0(xf(m))+x\cdot \theta(m).
\end{eqnarray*}
Thus we obtain
\begin{eqnarray}\label{eq:c1}
x\cdot \theta(m)+\omega(x,f(m))&=&\theta( x\cdot m )+\varphi(\mu(x,m)).
\end{eqnarray}

 Since $(\wg,\wal_{\g})$ is a Hom-associative algebra,
$$
\wal(\sigma_0(x))\big(\sigma_0(y)\sigma_0(z)\big)=\big(\sigma_0(x)\sigma_0(y)\big)\wal(\sigma_0(z)),
$$
then we get
\begin{equation}\label{eq:c3}
\omega(xy,\alin z)+\omega (x, y)\cdot \al(z)=\alin x\cdot \omega (y, z)+\omega (\al(x), yz).
\end{equation}
Since $(\wm,\wal_M)$ is a left module of $(\wg,\wal)$,  we have the equality
\begin{eqnarray*}
&&\wal\sigma_0(x)\cdot(\sigma_0(y)\cdot\sigma_1(m))-(\sigma_0(x)\sigma_0(y))\cdot(\wal_M\sigma_1(m))\\
&=&(\sigma_0\al(x))\cdot\big(\mu(y,m)+\sigma_1(y\cdot m)\big)\\
&&-\big(\omega(x,y)+\sigma_0(xy)\big)\cdot(\sigma_1\alm(m))\\
&=&\sigma_1(\al(x)\cdot (y\cdot m))+\mu(\al(x),y\cdot m)+\al(x)\cdot \mu(y,m)\\
%&&-\sigma_1(\rho(\al(y))\circ x\cdot m )-\mu(\al(y), x\cdot m )-\rho_{1}(\al(y))\mu(x,m)\\
&&-\sigma_1((xy)\cdot \alm(m))-\mu(xy,\alm(m))-\omega(x,y)\trr \alm(m)\\
&=&0.
\end{eqnarray*}
Thus, we get
\begin{eqnarray}
&&\al(x)\cdot\mu(y,m)+\mu(\al(x),y\cdot m)\nonumber\\
&&-\omega(x,y)\trr\alm(m)-\mu(xy,\alm(m))=0.\label{eq:c41}
\end{eqnarray}
Similarly, since $(\wm,\wal_M)$ is a right module of $(\wg,\wal)$,  we obtain
\begin{eqnarray}
&&\nu(m,x)\cdot \al(y)+\nu(m\cdot x,\al(y))\nonumber\\
&&-\alm(m)\trl\omega(x,y)-\nu(\alm(m), xy)=0.\label{eq:c42}
\end{eqnarray}
At last, since $(\wm,\wal_M)$ is a bimodule of $(\wg,\wal)$,  we have the equality
\begin{eqnarray*}
&&(\sigma_0(x)\cdot\sigma_1(m))\cdot \wal\sigma_0(y)-(\wal_M\sigma_0(x))\cdot(\sigma_1(m)\sigma_0(y))\\
&=&(\sigma_0\al(x))\cdot\big(\mu(y,m)+\sigma_1(y\cdot m)\big)\\
&&-\big(\omega(x,y)+\sigma_0(xy)\big)\cdot(\sigma_1\alm(m))\\
&=&\sigma_1(\al(x)\cdot (m\cdot y))+\mu(\al(x), m\cdot y)+\al(x)\cdot \nu(m,y)\\
&&-\sigma_1((x\cdot m)\cdot \al(y))-\mu(x\cdot m,\al(y))-\mu(x,m)\cdot \al(y)\\
&=&0.
\end{eqnarray*}
Thus, we get
\begin{eqnarray}
&&\mu(\al(x), m\cdot y)+\al(x)\cdot \nu(m,y))\nonumber\\
&&-\nu(x\cdot m, \al(y))-\mu(x,m)\cdot\al(y)=0.\label{eq:c43}
\end{eqnarray}

Therefore, by the above discussion, we obtain that $(\omega,\mu,\nu,\theta)$ is a $2$-cocycle of $(M,\g, f)$ with coefficients in $(V,W,\varphi)$.
This complete the proof.
\end{proof}

Now we define the Hom-associative algebra structure on $(M \oplus V, \g \oplus W, \widehat{f})$
 using the 2-cocycle given above. More precisely, we have
\begin{equation}\label{eq:multiplication}
\left\{\begin{array}{rcl}
\widehat{f}(m+v)&\triangleq&\dM(m)+\theta(v)+\varphi(v),\\
{(x+w)(x'+w')}&\triangleq&xx'+\omega(x,x')+x\cdot w'+w\cdot x',\\
{(x+w)\cdot (m+v)}&\triangleq& x\cdot m +\mu(x,m)+x\cdot v+ w\trr m,\\
{(m+v)\cdot (x+w)}&\triangleq& m\cdot x +\nu(m,x)+v\cdot x+ m\trl w,
\end{array}\right.
\end{equation}
for all $x,y,z\in{A}$, $m,n\in M$,
$v\in V$ and $w\in W$. Thus any extension
$\widehat{E}$ given by \eqref{eq:ext1} is isomorphic to
\begin{equation}\label{eq:ext2}
\CD
  0 @>0>>  V @>i_1>>  M \oplus V@>p_1>>  M  @>0>> 0 \\
  @V 0 VV @V \varphi VV @V \widehat{\dM} VV @V\dM VV @V0VV  \\
  0 @>0>>  W  @>i_0>> \g \oplus W @>p_0>> \g @>0>>0,
\endCD
\end{equation}
where the Hom-associative algebra structure in $\LM$  is given as above \eqref{eq:multiplication}.

\begin{thm}
There is a one-to-one correspondence between equivalence classes of abelian extensions and the elements of the second cohomology group $\mathbf{H}^2((M,\g, f),(V,W,\varphi))$.
\end{thm}

\begin{proof}
We have know from the above discussion that  abelian extensions of Hom-associative algebra in $\LM$ are correspond to 2-cocylce and vice verse.
Let $\E'$ be another abelian extension determined by the 2-cocycle $(\theta',\omega',\mu',\nu')$. We are going to show that $\E$ and $\E'$ are equivalent if and only if 2-cocycles  $(\omega,\mu,\nu,\theta)$ and $(\theta',\omega',\mu',\nu')$ are in the same cohomology class.

 Since $F$ is an equivalence of abelian extensions, there exist two linear maps $b_0:\g \longrightarrow W$
and $b_1: M \longrightarrow V$ such that
 $$F_0(x+w)=x+b_0(x)+w,\quad F_1(m+v)=m+b_1(m)+v.$$

First, by the equality
\begin{eqnarray*}
\label{eqn:fi} \widehat{f}'F_1(m)&=&F_0\widehat{f}(m),
\end{eqnarray*}
we have
\begin{equation}\label{eq:exact1}
\theta(m)-\theta'(m)=\varphi b_1(m)-b_0(\dM (m)).
\end{equation}

 Furthermore, we have
$$
F_0(xy+\omega(x,y))=F_0(x)F_0(y)
$$
which implies that
\begin{equation}\label{eq:exact2}
\omega(x,y)-\omega'(x,y)=x\cdot b_0(y)+b_0(x)\cdot y-b_0(xy).
\end{equation}

Similarly, by the equality
$$F_1(x\cdot m+\mu(x,m))=F_0(x)\cdot F_1(m)$$
we get
\begin{equation}\label{eq:exact3}
\mu(x,m)-\nu'(x,m)=x\cdot b_1(m)+b_0(x)\cdot m-b_1(x\cdot m).
\end{equation}

By \eqref{eq:exact1}-\eqref{eq:exact3}, we deduce that $(\psi,\omega,\mu,\nu)-(\psi',\omega',\mu',\nu')=D(b_0,b_1)$. Thus, they are in the same cohomology class.

Conversely, if  $(\omega,\mu,\nu,\theta)$ and $ (\theta',\omega',\mu',\nu')$ are in the same cohomology class,
assume that $(\omega,\mu,\nu,\theta)-(\theta',\omega',\mu',\nu')=D(b_0,b_1)$.
Then we define $(F_0,F_1)$ by
 $$F_0(x+w)=x+b_0(x)+w,\quad F_1(m+v)=m+b_1(m)+v.$$
Similar as the above proof, we can show that $(F_0,F_1)$ is an equivalence. We omit the details.
\end{proof}

\section*{Acknowledgements}
This work is supported in part by National Natural Science Foundation of China (11961049).
We would like to thank the referee for very helpful comments and suggestions for this paper.

\section*{Declarations}

 {\bf Competing interests} There are no conflicts of interest for this work.

\noindent {\bf Availability of data and materials} Data sharing is not applicable to this article as no datasets were generated or analyzed during the current study.

\end{document}